\documentclass{elsarticle}
\usepackage{amssymb,, amsmath, latexsym, enumerate}
\usepackage{graphicx}

\numberwithin{equation}{section}

\newtheorem{theorem}{Theorem}
\newtheorem{lemma}[theorem]{Lemma}
\newdefinition{remark}{Remark}
\newproof{proof}{Proof}
\newproof{pot}{Proof of Theorem \ref{thm2}}

\newcommand{\Z}{\mathbb{Z}}
\newcommand{\Q}{\mathbb{Q}}

\newcommand{\N}{\mathfrak{N}}
\newcommand{\n}{\mathfrak{n}}
\newcommand{\s}{\mathfrak{s}}
\newcommand{\gen}{\text{gen}}

\newcommand{\lan}{\langle}
\newcommand{\ran}{\rangle}

\DeclareMathOperator{\ord}{ord}
\DeclareMathOperator{\rad}{sf}

\begin{document}

\begin{frontmatter}

\title{A characterization of almost universal ternary quadratic polynomials with odd prime power conductor}
\author{Anna Haensch \corref{cor1}}

\cortext[cor1]{Corresponding Author: Anna Haensch, Duquesne University, Department of Mathematics and Computer Science, 600 Forbes Ave., Pittsburgh, PA, 15282, USA; Email, haenscha@duq.edu; Phone, +001 412 396 4851}


\begin{abstract}
An integral quadratic polynomial (with positive definite quadratic part) is called almost universal if it represents all but finitely many positive integers.  In this paper, we introduce the conductor of a quadratic polynomial, and give an effective characterization of almost universal ternary quadratic polynomials with odd prime power conductor.  
\end{abstract}

\begin{keyword}
primitive spinor exceptions \sep ternary quadratic forms 

\MSC[2010] 11E12 \sep11E20 \sep 11E25
\end{keyword}

\end{frontmatter}

\section{Introduction}
A fundamental question in the study of quadratic forms asks, given a quadratic polynomial $f$ and an integer $a$, how can we decide when $f$ represents $a$ over the integers?  This question has been the driving force behind many developments in number theory over the past centuries.  Very notably, in 1900, David Hilbert addressed the International Congress of Mathematicians in Paris posing his famous list of 23 problems; among them was the following:

\begin{quote}
To solve a given quadratic equation with algebraic numerical coefficients in any number of variables by integral or fractional numbers belonging to the algebraic realm of rationality determined by the coefficients \cite{Hil}.
\end{quote}

Along this line of inquiry, given an integer $a$ and an inhomogeneous quadratic polynomial 
\[
f(x)=Q(x)+\ell(x)+c,
\] 
in more than one variable, when is $f(x)=a$ solvable over the integers?  Here $Q$ and $\ell$ are homogeneous quadratic and linear polynomials, respectively, and $c$ is a constant.   Triangular numbers, defined by $T_x=\frac{x(x-1)}{2}$ where $x$ is an integer, give a particularly interesting family of inhomogeneous quadratic polynomials.  In 1796, Gauss observed that the sum of three triangular numbers, $T_x+T_y+T_z$, represents all natural numbers, a property which we call \emph{universal}.  This fact can be easily confirmed by completing the square, and noting that $8n+3$ is represented by the polynomial $(2x+1)^2+(2y+1)^2+(2z+1)^2$ as a consequence of the Three Square Theorem.  

In 1862 Liouville examined this problem further, by asking for which positive integer triples $(\alpha,\beta,\gamma)$ the weighted ternary sum $\alpha T_x+\beta T_y+\gamma T_z$ is universal.  He determined that only seven such triples exist, namely  $(1,1,1), (1,1,2), (1,1,4),$ $(1,1,5),(1,2,2),(1,2,3)$ and $(1,2,4)$.  Recently, in \cite{Guo07}, \cite{OS09}, and \cite{Sun07}, Z.-W. Sun et al.  effectively determine all $\alpha x^2+\beta T_y+\gamma T_z$ and $\alpha x^2+\beta y^2+\gamma T_z$ that are universal. 

An extension to the above problem is a characterization of quadratic polynomials representing all but finitely many natural numbers; that is, quadratic polynomials which are \emph{almost universal}.  Kane and Sun approach this question in \cite{KS10}, conjecturing a list of necessary and sufficient conditions on the triple $(\alpha,\beta,\gamma)$ of positive integers for which $\alpha T_x+\beta T_y+\gamma T_z$, $\alpha x^2+\beta T_y+\gamma T_z$ or $\alpha x^2+\beta y^2+\gamma T_z$ are almost universal.  In \cite{KS10}, they resolved the conjecture for $\alpha x^2+\beta y^2+\gamma T_z$ and prove sufficiency for the remaining two cases.  In \cite{CO09}, Chan and Oh resolve the conjecture for weighted sums of three triangular numbers, and in a joint paper with Chan \cite{CH11}, we resolve the remaining case. 

From here, the problem can be further generalized to arbitrary inhomogeneous quadratic polynomials, $f(x)=Q(x)+\ell(x)+c$, where $Q$ is a quadratic map associated to some positive definite lattice $N$, with symmetric bilinear map $B$.  Under the assumption that $N$ is positive definite, $\ell(x)=2B(\nu,x)$ for a unique choice of vector $\nu$ in $V:=\Q N$, and there is no harm in setting $c=0$ since the constant term does not contribute to the arithmetic of $f$.  An easy calculation reveals that an integer $n$ is represented by $f(x)$ if and only if $Q(\nu)+n$ is represented by the coset $\nu+N$.  

We define the conductor of the coset $\nu+N$ to be the smallest integer $\mathfrak{m}$ for which $\mathfrak{m}\nu\in N$.  Equivalently, $\mathfrak{m}=[M:N]$, where $M:=\Z\nu+N$.

As an example, consider the polynomial $f(x,y,z)=25x^2+y^2+z^2+10x$.   Taking $N\cong\lan25,1,1\ran$ in a basis $\{e_1,e_2,e_3\}$ and letting $\nu=\frac{e_1}{5}$, we see that $\nu$ has conductor 5, and $f(w)=Q(w)+2B(\nu,w)$ for any $w\in N$.  Now an integer $n$ is represented by $f$ if and only if $1+n$ is represented by the coset $\nu+N$; that is, $1+n=Q(\nu+e_1x+e_2y+e_3z))=(5x+1)^2+y^2+z^2$, which is just a sum of three squares with the added restriction that the first term is congruent to $1\mod 5$.  For this particular example it is easy to see that $f$ is not almost universal by the three square theorem, but this result will be further confirmed by Theorem \ref{main}.

When $N$ has rank greater than 3, Chan and Oh \cite[Theorem 4.9]{CO12} show how the asymptotic local-global principles for representations of lattices with approximation property by J\"ochner- Kitaoka \cite{JK94} and by Hsia-J\"ochner \cite{HJ97} lead to an asymptotic local-global principle for representations of integers by cosets, consequently we will restrict our attention to ternary $N$.  Imposing some mild arithmetic conditions on $f$, this paper establishes a characterization of almost universal ternary inhomogeneous quadratic polynomials.   

\section{Preliminaries}

Henceforth, the language of quadratic spaces and lattices as in \cite{OM} will be adopted, and the notation will follow that used in \cite{CH11}. Any unexplained notation and terminology can be found there and in \cite{OM}.   All the $\mathbb Z$-lattices discussed below are positive definite. If $K$ is a $\mathbb Z$-lattice and $A$ is a symmetric matrix, we shall write ``$K \cong A$" if $A$ is the Gram matrix for $K$ with respect to some basis of $K$. The discriminant of $K$, denoted $dK$, is the determinant of $A$.  An $n\times n$ diagonal matrix with $a_1, \ldots, a_n$ as the diagonal entries is written as $\lan a_1, \ldots, a_n\ran$.  However, we use the notation $[a_1, \ldots, a_n]$ to denote a quadratic space (over any field) with an orthogonal basis whose associated Gram matrix is the diagonal matrix $\lan a_1, \ldots, a_n\ran$.

The subsequent discussion involves the computation of the spinor norm groups of local integral rotations and the relative spinor norm groups of primitive representations of integers by ternary quadratic forms.   The formulae for all these computations can be found in \cite{EH75}, \cite{EH78}, \cite{EH94}, and \cite{H75}.  A correction of some of these formulae can be found in \cite[Footnote 1]{CO09}.  Following the notation set forth in \cite[\S 55]{OM}, we let $\theta$ denote the spinor norm map.  If $t$ is an integer represented primitively by $\gen(K)$ and $p$ is a prime, then $\theta^*(K_p, t)$ is the primitive relative spinor norm group of the $\Z_p$-lattice $K_p$. If $E$ is a quadratic extension of $\mathbb Q$, $\N_p(E)$ denotes the group of local norms from $E_{\mathfrak p}$ to $\mathbb Q_p$, where $\mathfrak p$ is an extension of $p$ to $E$.

Now suppose we have the integer-valued inhomogeneous polynomial as described above, where $Q$ is ternary.  For simplicity we will require that the $\Z$-ideal generated by $Q(\nu+x)$, for all $x\in N$, is contained in $\Z$.  Under this assumption, it is immediate that $Q(\nu)\in \Z$, and the $\Z$-ideal generated by $Q(x)+2B(\nu,x)$ for all $x\in N$, which we denote $\n(\nu,N)$, is contained in $\Z$.  It is well known that $Q(x)+2B(\nu,x)$ can never be almost universal, since this would imply that the ternary $\Z$-lattice $M$ is almost universal, which is impossible by the Local Square Theorem.  Therefore, a reasonable requirement is to fix a prime $p$, and assume that 
\begin{equation}\label{R1}
\n(\nu,N)=p^\alpha \Z, 
\end{equation}
with $\alpha>0$.  For the purposes of this paper, we will restrict out attention to odd $p$.  We define an integer-valued polynomial $H(x)$ by 
\[
H(x):=\frac{Q(x)+2B(\nu,x)}{p^\alpha},
\] 
where $p$ is an odd prime.  The sum of three triangular numbers can be obtained by letting $N\cong\lan4,4,4\ran$ in a basis $\{e_1,e_2,e_3\}$ with $\nu=\frac{e_1+e_2+e_3}{2}$, which corresponds to $p=2$ and $\alpha=3$.  While this is not the unique setup for a sum of three triangular numbers, any choice of $N$ and $\nu$ will leave $p=2$, and therefore the sum of triangular numbers will not be included in this discussion.  

For the remainder of this paper, we will assume that equation (\ref{R1}) holds, noting that under this assumption $B(\nu,N)$ and $Q(N)$ are subsets of $p^\alpha\Z$.  We also note that under these conditions there is a unique pair $(\nu,N)$ associated to $H(x)$, and without confusion we will let $\mathfrak{m}$ denote the conductor associated to this pair.  Our goal is to obtain a list of necessary and sufficient conditions on the pair $(\nu,N)$ under which this $H(x)$ is almost universal.  

\begin{lemma}\label{L1}
\begin{enumerate}[(1)]
\item If $H(x)$ is almost universal, then $M_q$ represents all $\Z_q$ whenever $q\neq p$, and $N_q$ represents all $\Z_q$ whenever $q\nmid \mathfrak{m}$. 
\item If $N_q$ represents all of $\Z_q$ then $\theta(O^+(N_q))\supseteq \Z_q^\times$. 
\item If $M_q$ represents all $\Z_q$ for all $q\neq p$, then $\Q_p M=\Q_p N$ is anisotropic. 
\end{enumerate}
\end{lemma}

\begin{proof}
(1) Suppose $H(x)$ is almost universal.  Then $Q(\nu)+p^\alpha n$ is represented by the coset $\nu+N$ for all but finitely many positive integers $n$.  Consequently, all but finitely many positive integers are represented by $\nu+N_q$ for $q\neq p$, and therefore by $M_q$ for $q\neq p$.  Now by the denseness of $\Z$ in $\Z_q$, $M_q$ represents all of $\Z_q$.  And for $q\nmid \mathfrak{m}$, $\nu+N_q=N_q$, so it follows that $N_q$ represents all of $\Z_q$. 

Part (2) follows immediately from the definition of spinor norm map.  Since $M$ is ternary, it is not difficult to show that if $M_q$ is universal then $M_q$ is isotropic.  Armed with this, part (3) is a direct consequence of Hilbert Reciprocity.  
$\square$ \end{proof}

In view of Lemma \ref{L1}, it will be helpful to impose some restriction on the prime divisors of the conductor of $H(x)$.  A reasonable assumption is that 
\begin{equation}\label{R2}
\mathfrak{m}=p^\gamma
\end{equation} 
for some positive integer $\gamma$.  Throughout the remainder of this paper both (\ref{R1}) and (\ref{R2}) will be assumed. 

As an immediate consequence of Lemma \ref{L1}, if $H(x)$ is almost universal, then $\ord_p(Q(\nu))<\alpha$.  In view of this fact, it is possible to scale $Q$ by $p^{-\ord_p(Q(\nu))}$ and then assume that $\ord_p(Q(\nu))=0$.  For simplicity, we define $\epsilon:=Q(\nu)\in \Z_p^\times$.  Now, $B(\nu,p^\gamma\nu)=p^\gamma \epsilon\equiv 0\mod p^\alpha$, and therefore we conclude that $\gamma\geq \alpha$. 

An arbitrary coset of $N$ in $M$ is of the form $a\nu+N$, where $0\leq a\leq p^\gamma-1$.  Given $x\in N$, we have $Q(a\nu+x)\equiv a^2\epsilon\mod p^\alpha$.  Hence, if $\epsilon+p^\alpha n$ is represented by this $a\nu+N$, then $\epsilon\equiv a^2\epsilon\mod p^\alpha,$ and therefore $a\equiv \pm 1\mod p^\alpha$, because $p$ is odd.  From this we conclude that every representation of $\epsilon+p^\alpha n$ by $M$ must be from one of the cosets $(bp^\alpha\pm 1)\nu+N$, $0\leq b\leq p^{\gamma-\alpha}$.  In order to be sure that a representation of $\epsilon+p^\alpha n$ is from $\nu+N$, we need to make an additional assumption; namely, $\gamma=\alpha$.  Under this assumption, any representation of $\epsilon+p^\alpha n$ by $M$ must be from $\nu+N$ or $-\nu+N$.  However, it is not difficult to see that these cosets represent precisely the same integers.   Under this assumption, (\ref{R2}) becomes $\mathfrak{m}=p^\alpha$.

Given a primitive spinor exception $t$ of the genus of $M$, we define $E:=\Q(\sqrt{-tdM})$.  Note that if $t$ is a primitive spinor exception of $\gen(M)$, then $-tdM\not\in \Q^2$, as shown in \cite{K56}.

\begin{lemma}\label{L2}
Suppose that $M_q$ represents all $q$-adic integers for every prime $q\neq p$.  If $t$ is a primitive spinor exception of $\gen(M)$, then $E=\Q(\sqrt{-p})$ and $p\equiv 7\mod 8$, when $p>2$.  
\end{lemma}

\begin{proof}
Suppose that $M_q$ represents all $q$-adic integers for every prime $q\neq p$.  For any $q\neq p$, $M_q=N_q$, and so $\Z_q^\times\subseteq \theta(O^+(M_q))\subseteq \mathfrak{N}_p(E)$, by Lemma \ref{L1} part (2).  It is immediate that $E$ is unramified at all $q\neq p$, and therefore no prime other than $p$ will divide the discriminant of $E$. 

For odd primes $p$, there is only one choice of an imaginary quadratic extension ramified only at the prime $p$, namely, $\Q(\sqrt{-p})$, where $p\equiv 3\mod 4$.  By our initial assumption $M_2$ represents all integers in $\Z_2$, so from the definition of the spinor norm, it is clear that $\mathfrak{N}_2(E)\supseteq \theta(O^+(M_2))=\Q_2^\times$.  Consequently, $[\Q_2^\times:\mathfrak{N}_2(E)]=1$, and hence $\Q_2(\sqrt{-p})=\Q_2$.  Therefore, $p\equiv 7\mod 8$.  
$\square$ \end{proof}

In the following lemma, we show that if $\nu$ is replaced by $\nu+x_0$, where $x_0\in N$, then all results on almost universality will still hold.  

\begin{lemma}\label{L3}
Let $\omega=\nu+x_0$, where $x_0\in N$, and define
\[
H'(x)=\frac{1}{p^\alpha}[Q(x)+2B(\omega,x)].
\]
Then, 
\begin{enumerate}[(1)]
\item if $\n(\nu,N)=p^\alpha \Z$, then $\n(\omega, N)=p^\alpha \Z$; and,
\item if $H(x)$ is almost universal, then $H'(x)$ is almost universal.  
\end{enumerate}
\end{lemma}

\begin{proof}
Suppose that $\n(\nu,N)=p^\alpha \Z$.  Let $\omega=\nu+x_0$ with $x_0\in N$.  First, observe that for any $x\in N$,
\[
Q(x+x_0)+2B(\nu, x+x_0) = Q(x)+Q(x_0)+2B(x,x_o)+2B(\nu,x)+2B(\nu,x_0),
\]
which implies that $2B(x,x_0)\equiv 0\mod p^\alpha$ for any $x\in N$.  Now, for any $x\in N$ 
\[
Q(x)+2B(\omega,x)=Q(x)+2B(\nu,x)+2B(x_0,x)\equiv 0\mod p^\alpha,
\]
and hence $\n(\omega, N)\subseteq p^\alpha \Z$.  Switching the roles of $\nu$ and $\omega$ in the argument above, we get containment in the other direction.  Therefore, $\n(\omega,N)=p^\alpha \Z$

Define $p^\alpha m:=Q(x_0)+2B(\nu,x_0)$.  Then, $Q(\omega)+p^\alpha (n-m)=Q(\nu)+p^\alpha n$.  If $H(x)$ is almost universal, then $\nu+N$ represents $Q(\nu)+p^\alpha n$ for $n$ sufficiently large, and hence $\nu+N=\omega+N$ represents $Q(\omega)+p^\alpha (n-m)$ for all $n$ sufficiently large.  Therefore, if $H(x)$ is almost universal, then $H'(x)$ is almost universal.  $\square$
\end{proof}

Let $\{e_1,e_2,e_3\}$ be a basis for $N$. In view of Lemma \ref{L3}, we may assume that $\nu=\frac{1}{p^\alpha}(ae_1+be_2+ce_3)$, with $0\leq a,b,c\leq p^\alpha-1$, and at least one of $a,b,c$, not divisible by $p$.  There is no harm in assuming that $a\in \Z_p^\times$.  It is clear that all $\Z_p$-linear combinations of $\nu,e_2,$ and $e_3$ are contained in $M_p$, including $e_1=\frac{1}{a}(p^\alpha\nu-be_2-ce_3)$.  Therefore, $\{\nu,e_2,e_3\}$ is a basis for $M_p$, and relative to this basis every entry in the Gram matrix for $M_p$ is congruent to 0 $\mod p$, with the notable exception of $\epsilon$.  Therefore, as a consequence of the uniqueness of Jordan invariants, $M_p$ must be of the form 
\[
M_p\cong\lan\epsilon, p^i\beta, p^j\gamma\ran
\]
in a basis $\{\nu,f_1,f_2\}$, with $\beta,\gamma\in \Z_p^\times$ and $0\leq i\leq j$.   

We conclude this section with several technical lemmas which will be very helpful in proving the main result.  

\begin{lemma}\label{LA1}
If $K$ is a universal ternary $\Z_2$-lattice, then $K$ is isotropic and $\ord_2(dK)<2$.  When this is the case, $K\cong \lan 1,-1,-dK\ran$ and 
\begin{enumerate}[(1)]
\item if $\ord_2(dK)=1$, then $K$ primitively represents every element of $\Z_2$ except those in $4\Z_2^\times$;
\item if $\ord_2(dK)=0$, then $K$ primitively represents every element of $\Z_2$. 
\end{enumerate}
\end{lemma}

\begin{proof}
Suppose that $K$ is a universal ternary $\Z_2$-lattice, then in particular $K$ must be isotropic.

Since $K$ is universal, $K$ has $\lan1\ran$ is an orthogonal summand, and therefore $K$ is of the form $\lan1\ran\perp L$, where $L$ is a binary $\Z_2$-lattice, which is either proper or improper.  If $L$ is proper then $K$ has an orthogonal decomposition, and if $L$ is improper then it is of the form $2^kA(2,2)$ given in \cite[\S 93B]{OM}, where $k>0$.  We note that if $k=0$ then $L$ has an orthogonal basis, so scaling by the first term we can obtain something of the form $\lan1\ran\perp L$ with $L$ proper.  

First we will suppose that $L$ is proper with $\n(L)=\Z_2$, and therefore $K$ has an orthogonal decomposition of the form $\lan 1,\beta, 2^j\gamma\ran$ in a basis $\{e_1,e_2,e_3\}$ with $\beta,\gamma\in \Z_2^\times$ and $j>0$.  If $j\geq 2$, then for any arbitrary choice of $\beta$ and $\gamma$ it can easily be shown that $K$ will either miss some square class in $\Z_2^\times$ or every element in $2\Z_2^\times$.  Therefore, we conclude that $j<2$ when $K$ is universal. 

Suppose that $j=0$ and consequently $K\cong\lan1,\beta,\gamma\ran$.  Then $K$ is an isotropic unimodular $\Z_2$-lattice, and is therefore of the form $\mathbb H\perp \lan -\beta\gamma\ran$.  From this decomposition it is easy to verify that $K$ primitively represents every element in $\Z_2$.   

Suppose that $j=1$.  Since $K$ is isotropic, and the binary component $\lan\beta,2\gamma\ran$ is $\Z_2$-maximal, it follows that $\lan \beta,2\gamma\ran\cong\lan-1,-2\beta\gamma\ran$, and thus $K\cong\lan1,-1,-2\beta\gamma\ran$ in a basis $\{f_1,f_2,f_3\}$.  It is clear that the leading binary component of  $K$ primitively represents every element in $\Z_2^\times$.  Furthermore, for any $\delta\in \Z_2$, the vector $(1+\delta)f_1+(1-\delta)f_2$ is of length $4\delta$, and $(1+\delta)f_1+(1-\delta)f_2+f_3$ is of length $4\delta-2\beta\gamma$.  We observe that $(1+\delta)f_1+(1-\delta)f_2+f_3$ is a primitive vector for any choice of $\delta$, and $(1+\delta)f_1+(1-\delta)f_2$ is primitive just in case $\delta$ is even.  Hence we have shown that $K$ primitively represents every element of $\Z_2$, with the exception of those in $4\Z_2^\times$. 

Suppose that $\delta\in \Z_2^\times$.  Then there exist $x,y,z\in \Z_2$ such that $4\delta=x^2-y^2-2\beta\gamma z^2$.  It is clear that $x$ and $y$ have the same parity, or else the right hand side of the equality is odd.  But then, $x^2-y^2\equiv(x+y)(x-y)\equiv 0\mod 4$, meaning $z$ must be even, and hence $4\delta\equiv x^2-y^2\mod 8$.  If $x$ and $y$ were both even, then the representation would not be primitive, as such, we assume that $x$ and $y$ are units in $\Z_2$, which means that $x^2\equiv y^2\equiv 1\mod 8$.  Consequently, $4\delta\equiv 0\mod 8$, which is not true, since $\delta\in \Z_2^\times$.  Therefore, any representation of elements from $4\Z_2^\times$ by $K$ must be imprimitive.  
 
Suppose the unimodular component of $K$ is unary; that is, $K\cong\lan1\ran\perp 2^i L$, where $i>0$ and $\s(L)=\Z_2$.  If either $L$ is improper or if $i>1$, then $K$ only represents units congruent to $1\mod 4$, so we may assume that $i=1$ and $L\cong\lan\beta, 2^j\gamma\ran$, with $j\geq 0$.  Now, $K\cong\lan1,2\beta,2^{j+1}\gamma\ran$ and therefore we may assume that $j=0$ or $1$.  

Suppose that $j=0$, and $K\cong\lan 1,2\beta,2\gamma\ran$.  Then the only units represented by $K$ are those congruent to $1$, $1+2\beta$, $1+2\gamma$ or $1+2(\beta+\gamma)\mod 8$.  For any choice of $\beta$ and $\gamma$, it is clear that $K$ fails represent one square class of units, and hence it is not universal.  

Suppose that $j=1$, and $K\cong\lan 1,2\beta,4\gamma\ran$.  Since $K$ is isotropic, and since its leading binary component is $\Z_2$-maximal,  this implies that $K\cong\lan-\gamma,-2\beta\gamma,4\gamma\ran$, and hence $K^{-\gamma}\cong\lan 1,2\beta, -4\ran$.  Suppose that $\eta\in \Z_2^\times$, with $\eta\equiv \beta\mod 4$ and $\eta\not\equiv \beta\mod 8$.  If $2\eta$ is represented by $K^{-\gamma}$, then $\eta\equiv 2(x+y)(x-y)+\beta\mod 8$ for some $x,y\in \Z_2$.  If $x+y$ is even, then this implies that $\eta-\beta\equiv 0\mod 8$, a contradiction.  On the other hand, if $x+y$ is odd, then this implies that $\frac{\eta-\beta}{2}\equiv 1,3\mod 4$, also a contradition.  Therefore, $K^{-\gamma}$ does not represent $2\eta$, and hence $K$ is not universal.  
$\square$ \end{proof}

\begin{lemma}\label{shapeofnp}
If $H(x)$ is almost universal, then $N_q$ represents all of $\Z_q$ whenever $q\neq p$, and consequently, $N_q\cong\lan 1,-1,-dN\ran$ for $q\neq p,2$, and $N_2\cong\lan 1,-1,-dN\ran$ with $\ord_2(dN)\leq 1$.  
\end{lemma}

\begin{proof}
If $H(x)$ is almost universal, then $N_q$ represents all of $\Z_q$ whenever $q\neq p$ by Lemma \ref{L1} part (1).  Hence, Lemma \ref{LA1} implies that that $N_2\cong\lan 1,-1,-dN\ran$ with $\ord_2(dN)\leq 1$.  And for odd primes $q\neq p$, since $N_q$ is universal it must also be isotropic, and therefore it follows that $N_q\cong\lan 1,-1,-dN\ran$.
$\square$ \end{proof}

\begin{lemma}\label{L3.4}
Suppose that $N_q$ represents every integer in $\Z_q$ for all primes $q\neq p$.  Then, every positive integer $\epsilon+p^\alpha n$ which is not of the form $4\delta$, where $\delta\in \Z_2^\times$, is represented primitively by $\gen(M)$.  If $\epsilon+p^\alpha n=4\delta$, then $\frac{\epsilon+p^\alpha n}{4}$ is represented primitively by $\gen(M)$. 

\end{lemma}

\begin{proof}
First we observe that $\epsilon\in \Z_p^\times$ is clearly represented primitively by $M_p$, so any element in $\Z_p$ which is in the same square class modulo $p$ as $\epsilon$ will also be represented primitively by $M_p$.  

For all odd primes $q\neq p$, it follows from Lemma \ref{shapeofnp} that every $q$-adic integer is primitively represented by $N_q=M_q$.  And it is immediate from Lemma \ref{LA1}  $\epsilon+p^\alpha n$ (or $\frac{\epsilon+p^\alpha n}{4}$, accordingly) is represented primitively by $N_2=M_2$. 

Hence, all $\epsilon+p^\alpha n$ (and $\frac{\epsilon+p^\alpha n}{4}$ as appropriate) are represented primitively by $\gen(M)$. $\square$ 
\end{proof}

\section{Main Result}

Let $\rad(dN)$ denote the square-free part of the discriminant of $N$, and let $\rad(dN)'$ denote the non-$p$ part of $\rad(dN)$. 

\begin{theorem}\label{main}
$H(x)$ is almost universal if and only if $N_q$ represents all $q$-adic integers over $\Z_q$ whenever $q\neq p$, and 
\begin{enumerate}[(1)]
\item $p\not\equiv 7\mod 8$; or,
\item  $p\equiv 7\mod 8$ and one of the following holds:
\begin{enumerate}[(a)]
\item $\ord_p(dN)$ is even;
\item $\rad(dN)$ is divisible by a prime $q$ satisfying $\big(\frac{-p}{q}\big)=-1$;
\item  $\big(\frac{Q(\nu)}{p}\big)=-1$;
\item $\rad(dN)'$ is represented by $\nu+N$.
\end{enumerate}
\end{enumerate}
\end{theorem}

\begin{proof}
We suppose throughout that $N_q$ represents all $q$-adic integers over $\Z_q$ whenever $q\neq p$. 

For part (1), suppose that $N_q$ is universal over $\Z_q$ for every $q\neq p$, and $p\not\equiv 7\mod 8$.  From Lemma \ref{L2}, it is clear that $M$ has no primitive spinor exceptions.  Combining this with Lemma \ref{L3.4}, $\epsilon+p^\alpha n$ (or $\frac{\epsilon+p^\alpha n}{4}$, as appropriate) is represented primitively by the spinor genus of $M$ for all $n$.  Now, it follows from \cite[Corollary]{DS-P90} that $\epsilon+p^\alpha n$ (similarly $\frac{\epsilon+p^\alpha n}{4}$) is represented by the lattice $M$ for $n$ sufficiently large, and hence by the coset $\nu+N$.  Therefore, $H(x)$ is almost universal. 
 
For part (2), we suppose that $p\equiv 7\mod 8$.  If $\epsilon+p^\alpha n$ is a primitive spinor exception of the genus of $M$, then $E:=\Q(\sqrt{-(\epsilon+p^\alpha n)dN})$ is $\Q(\sqrt{-p})$, by Lemma \ref{L2}.  But $\ord_p(dN)$ and $\ord_p(dM)$ have the same parity, and in this case, they must both be odd.  Therefore, if $\ord_p(dN)$ is even, then $\epsilon+p^\alpha n$ cannot be a primitive spinor of exception of the genus of $M$.  

Suppose that (a) fails, so $\ord_p(dN)$ is odd.  Suppose that $q$ is a prime with $\big(\frac{-p}{q}\big)=-1$.  At such a prime, $-p$ is a non-square, and hence $E_q=\Q_q(\sqrt{-p})$ is a quadratic extension over $\Q_q$.  By \cite[Theorem 1(a)]{EH94}, $\theta(O^+(M_q))\subseteq \mathfrak{N}_q(E)$ if and only if $\ord_q(dN)$ is even.  If (b) holds, then $q\mid \rad(dN)$ meaning that $\theta(O^+(M_q))\not\subseteq \mathfrak{N}_q(E)$, and therefore $\epsilon+p^\alpha n$ is not a primitive spinor exception of the genus of $M$.  

Suppose (a) and (b) fail, but (c) holds.  Then, 
\[
M_p\cong\lan\epsilon, p^i\beta, p^j\gamma\ran
\]
and since $\ord_p(dN)$ is odd, $i$ and $j$ have different parity.  Now we can compute the spinor norm of $O^+(M_p)$ using \cite[Satz 3]{K56}.  Suppose that $i$ is even, meaning that $\Q_pM$ contains $[\epsilon,\beta]$ as a subspace.  Since $M_p$ is anisotropic by Lemma \ref{L1} part (3), we conclude that $-\epsilon\beta$ must be a nonsquare in $\Q_p$.  Furthermore, since $p\equiv 7\mod 8$, and therefore $-1$ is not a square in $\Q_p$, we conclude that $\epsilon\beta$ is a square in $\Q_p$.  Since $\rad(dN)'$ is only divisible by primes which are squares in $\Q_p$ by the failure of (b), it follows that $\gamma$ is also a square in $\Q_p$.  Computing the spinor norm in this case gives
\[
\theta(O^+(M_p))=\{1, \epsilon\beta,\epsilon\gamma p, \beta\gamma p\}\Q_p^{\times 2}=\{1,\epsilon p\}\Q_p^{\times 2},
\]
where by (c), $\epsilon$ is non-square in $\Q_p$.  For $E$ defined as above, $\mathfrak{N}_p(E)=\{1,p\}\Q_p^{\times 2}$.  Since, $\theta(O^+(M_q))\not\subseteq \mathfrak{N}_q(E)$, therefore $\epsilon+p^\alpha n$ cannot be a primitive spinor exception of $\gen(M)$.  The case when $j$ is even is the same, substituting $i$ for $j$ and $\beta$ for $\gamma$ where necessary.  

We have shown above that when one of (a), (b), or (c) holds, then $\gen(M)$ has no primitive spinor exceptions of the form $\epsilon+p^\alpha n$ in $\Q_p$.  In fact, since these arguments rely only on the square class of $\epsilon+p^\alpha n$, they can be easily extended to show that $\gen(M)$ has no primitive spinor exceptions which are of the form $t^2(\epsilon+p^\alpha n)$, where $t\in \Z_p^\times$.  Suppose we are in the special case where $\epsilon+p^\alpha n=4\delta$ with $\delta\in \Z_2^\times$.  Then $\frac{\epsilon+p^\alpha n}{4}$ is represented primitively by the genus of $M$ by Lemma \ref{L3.4}, and if one of (a), (b), or (c) holds, then $\frac{\epsilon+p^\alpha n}{4}$ is not a primitive spinor exception of the genus of $M$.  From \cite[Corollary]{DS-P90}, we conclude that $\frac{\epsilon+p^\alpha n}{4}$, and hence $\epsilon+p^\alpha n$, are represented by $M$, for all sufficiently large $n$.  Therefore, $\epsilon+p^\alpha n$ are represented by $\nu+N$ for all sufficiently large $n$.   In a similar manner, when we are not in this special case, then $\epsilon+p^\alpha n$ is represented primitively by the genus of $M$ by Lemma \ref{L3.4}, and if one of (a), (b), or (c) holds, then $\epsilon+p^\alpha n$ is not a primitive spinor exception of the genus of $M$ for every positive integer $n$.  Hence, $\epsilon+p^\alpha n$ is represented by $M$ for all $n$ sufficiently large, and therefore $H(x)$ is almost universal. 

Suppose (a), (b), and (c) all fail.  We will show that $\rad(dN)'$ is a primitive spinor exception of $\gen(M)$.  Without any confusion, in what follows, let $E$ denote the field $\Q(\sqrt{-\rad(dN)'dN})=\Q(\sqrt{-p})$.  

First, we claim that $\rad(dN)'$ is represented primitively by the genus of $M$.  Since the prime factors of $\rad(dN)'$ can appear with order at most 1, it is clear from combining Lemma \ref{LA1} and Lemma \ref{shapeofnp} that $\rad(dN)'$ is represented primitively by $M_q$ when $q\neq p$.  By the failure of (b), $\rad(dN)'$ is only divisible by primes which are squares in $\Q_p$, and therefore $\rad(dN)'$ is a square unit in $\Q_p$.  By the failure of (c), $\epsilon$ is a square unit in $\Q_p$.  This means that $\rad(dN)'$ and $\epsilon$ are in the same square class modulo $p$.  Since $\epsilon$ is represented primitively by $M_p$, it follows that $\rad(dN)'$ is also represented primitively by $M_p$.  This proves that $\rad(dN)'$ is represented primitively by $\gen(M)$. 

Since $p\equiv 7\mod 8$, $E$ splits at the prime $2$.  Moreover, $E$ also splits at all primes $q$ such that $\big(\frac{-p}{q}\big)=1$.  Hence for these primes $\theta(O^+(M_q))\subseteq \mathfrak{N}_q(E)=\Q_q^\times$ for these $q$, and therefore $\theta^*(M_q,\rad(dN)')=\mathfrak{N}_q(E)$, as explained in equation (3) of \cite{EH94}.  At those primes $q$ for which $\big(\frac{-p}{q}\big)=-1$, it follows from the failure of (b) that $q\nmid \rad(dN)'$, and $M_q\cong\lan 1,-1,dN\ran$ from Lemma \ref{shapeofnp}.  Therefore, from \cite[Theorem 1(a)]{EH94}, we have $\theta(O^+(M_q))\subseteq \mathfrak{N}_q(E)$, and since $\rad(dN)'$ is not in any ideal generated by $q$ we get $\theta^*(M_q,\rad(dN)')=\mathfrak{N}_q(E)$.  

At the prime $p$ we will compute the spinor norm using \cite[Satz 3]{K56}.  Since 
\[
M_p\cong\lan \epsilon,\beta p^i,\gamma p^j\ran,
\]
we get 
\[
\theta(O^+(M_p))=\{1, \epsilon\beta p^i,\epsilon \gamma p^j, \beta\gamma p^{i+j}\}\Q_p^{\times 2}.
\]
By the failure of (c), $\epsilon$ is a square in $\Q_p$.  If $i$ is even, then $\epsilon\beta$ must be a square in $\Q_p$ since $\Q_pM$ contains the anisotropic subspace $[\epsilon,\beta]$.  In this case, $\gamma$ must also be a square in $\Q_p$, by the failure of (b).  Similarly, if $i$ is odd, and $j$ is even, then $\gamma$ is a square in $\Q_p$, since $\Q_pM$ contains the anisotropic subspace $[\epsilon,\gamma]$, and $\epsilon$ is a square.  Therefore, in any case
\[
\theta(O^+(M_p))=\{1, p\}\Q_p^{\times 2},
\]
and hence $\theta(O^+(M_p))\subseteq \mathfrak{N}_p(E)$.  Finally, $\mathfrak{N}_p(E)=\theta^*(M_p,\rad(dN)')$, by \cite[Theorem 1(b)]{EH94}, since $\rad(dN)'$ will not be in any ideal generated by $p$.

Suppose that (a), (b), and (c) all fail.  Now, we will show that $H(x)$ is almost universal if and only if (d) holds.  Suppose that (d) holds, then $\rad(dN)'$ is represented by $M$.  Suppose we are in the special case $\epsilon+p^\alpha n=4\delta$ with $\delta\in \Z_2^\times$.  From Lemma \ref{L3.4}, we know that $\delta$ is represented primitively by the genus of $M$.  Suppose that $\delta$ is not a primitive spinor exception of the genus of $M$.  Then, all sufficiently large $\delta$ and hence $4\delta$ are represented by $M$.  Suppose that $\delta$ is a primitive spinor exception, then $\delta$ is a square multiple of $\rad(dN)'$ and hence $\delta$ and therefore $4\delta$ is represented by $M$.  Thus $4\delta$ must be represented by $M$.  Now suppose that $\epsilon+p^\alpha n\not\in 4\Z_2^\times$, and suppose $\epsilon+p^\alpha n$ is not a primitive spinor exception of $\gen(M)$.  Then $\epsilon+p^\alpha n$ is represented by $M$ when $n$ is sufficiently large.  Otherwise, $\epsilon+p^\alpha n$ must be a square multiple of $\rad(dN)'$ and hence $\epsilon+p^\alpha n$ represented by $M$.  Therefore, $\epsilon+p^\alpha n$ is represented by $M$ for all $n$ sufficiently large and hence $H(x)$ is almost universal.  

Conversely, suppose that (d) fails.  This implies that $\rad(dN)'$ is not represented by $M$.  Then there exist, as shown in \cite{S-P00}, infinitely many primes $q$ such that $q^2\rad(dN)'$ are not represented by $M$.  Further, these $q$ are precisely the primes that split in $\Q(\sqrt{-p})$.  Let $\mu$ be the inverse of $\rad(dN)'$ modulo $p$.  Then $\epsilon\mu$ is a square modulo $p^\alpha$, and thus $\rho^2\equiv \epsilon\mu\mod p^\alpha$ for some integer $\rho$.   

Since $p\equiv 7\mod 8$, the cyclotomic extension $\Q(\zeta_{p^\alpha})$ contains $\Q(\sqrt{-p})$ as its unique quadratic extension.  For an odd prime $q\neq p$, let $J$ be the fixed field of the decomposition group, $G_q$, which is generated by the Frobenius automorphism $\varphi_q$.  Then, $J$ is the maximal subfield of $\Q(\zeta_{p^\alpha})$ in which $q$ splits completely.  Since $q$ splits completely in $J$ if and only if $q$ splits completely in any subextension of $J$ properly containing $\Q$, Galois theory implies that $q$ splits completely in $\Q(\sqrt{-p})$ if and only if $\varphi_q$ is in $H:=Gal(\Q(\zeta_{p^\alpha})/\Q(\sqrt{-p}))$.

The $2$-element subgroup generated by complex conjugation, $\varphi_{-1}$, has as its fixed field the maximal real subfield $\Q(\zeta_{p^\alpha})^+$, which cannot sit between $\Q(\sqrt{-p})$ and $\Q(\zeta_{p^\alpha})$.  For any integer $\rho$ relatively prime to $p$, either $\varphi_\rho\in H$, or $\varphi_\rho\varphi_{-1}=\varphi_{-\rho}\in H$.  If $\varphi_\rho\in H$, then by the $\check{C}$ebotarev Density Theorem there exist infinitely many primes $q$ such that $\varphi_q=\varphi_\rho\in H$, and so $q$ splits in $\Q(\sqrt{-p})$.  On the other hand, if $\varphi_\rho\not\in H$, then $\varphi_{-\rho}\in H$, and again there exist infinitely many primes $q$ such that $\varphi_q=\varphi_{-\rho}\in H$.  In both cases, $q$ splits in $\Q(\sqrt{-p})$ and either $\rho\equiv q\mod p^\alpha$, or $-\rho\equiv q\mod p^\alpha$.    

Therefore, for a fixed integer $\rho$, we know that there are infinitely many primes $q$ that split in $\Q(\sqrt{-p})$ and either $q\equiv \rho\mod p^\alpha$ or $q\equiv-\rho\mod p^\alpha$.  So we have an infinite family of positive integers, $n:=\frac{q^2\rad(dN)'-\epsilon}{p^\alpha}$, for which $\epsilon+p^\alpha n$ are not represented by $M$, and hence $H(x)$ is not almost universal.   
$\square$ \end{proof}

\section*{Acknowledgements}

The author would like to thank her thesis advisor Wai Kiu Chan for his indispensable guidance in this project.  



\end{document}